\documentclass[12pt]{article}
\usepackage{amsfonts}
\usepackage{amsmath}
\usepackage{setspace}
\usepackage[top=3cm, bottom=3cm, left=3cm, right=3cm]{geometry}
\newenvironment{proof}[1][Proof]{\textbf{#1 }}{$\square$}
\newtheorem{theorem}{Theorem}

\newtheorem{definition}[theorem]{Definition}

\newtheorem{proposition}[theorem]{Proposition}

\begin{document}

\title{SEIBERG-WITTEN LIKE EQUATIONS ON $5-$DIMENSIONAL CONTACT METRIC MANIFOLDS\footnote{This paper was supported by the project 1105F099 of the Scientific Research Foundation of Anadolu University.}
}
\date{}
\author{Nedim DE\v{G}\.{I}RMENC\.{I} and \c{S}enay BULUT\\ Department of Mathematics, Science Faculty, Anadolu University,\\  Eskisehir, Turkey\\[0pt]
\texttt{ndegirmenci@anadolu.edu.tr, skarapazar@anadolu.edu.tr}\\[0pt]}
\maketitle
\begin{abstract}
In this paper, we write down Seiberg-Witten equations on contact metric manifolds of dimension $5$. Any contact metric manifold has a spin$^c-$structure. For Dirac equation we use Dirac type operators associated to the generalized Tanaka-Webster connection on spin$^c$ spinor bundle of a contact metric manifold. For curvature equation we need to self-duality concept. Self-duality concept is significant on odd dimensional manifolds, particularly, on 5-dimensional contact manifolds. Finally, we give a global solution to these equations on strictly pseudoconvex CR manifolds.

\end{abstract}

\textbf{2000 MR Subject Classification} 15A66, 58Jxx. \label{first}

\section{Introduction}
Seiberg-Witten equations are defined on four dimensional Riemannian manifolds by Witten in \cite{Witten}. Solution space of these equations gives differential topological invariants for 4-manifolds \cite{Akbulut,Morgan}. On the other hand, there are some generalizations of Seiberg-Witten monopole equations to higher dimensional manifolds \cite{BDK1, ND, Tian2}.

 Seiberg Witten equations consist of two equations. First one is the Dirac equation, to able to write this equation the manifold must have spin$^c-$structure. The second one is the curvature equation which couples the self-dual part of a connection 2-form with a spinor field. To able to write the curvature equation the self-duality notion of a $2-$form is needed. This notion is meaningful for $4-$dimensional Riemannian manifolds. On the other hand, there are similar self-duality notions for some higher dimensional manifolds \cite{Corigan,Tian1}. In the present paper we propose Seiberg-Witten like equations for $5-$dimensional contact metric manifolds. Because contact metric manifolds have spin$^c-$structure \cite{petit} and a kind of self-duality notion of a $2-$form was defined on contact metric $5-$manifolds \cite{Bellettini}. A contact metric structure on a contact manifold is a tuple $(M,g_{\eta},\eta,\xi,J)$ where $\eta$ is a contact form, $J$ is an almost complex structure on the contact distribution $H=ker\ \eta$ and $g_{\eta}$ is a Riemannian metric compatible with both almost complex structure and the contact form such that $$g_{\eta}(X,Y)=d\eta(X,JY)$$for any $X,Y\in \Gamma(H)$. A contact metric manifold is a contact manifold equipped with the contact metric structure. Any contact metric manifold has a canonical spin$^c-$structure associated to its contact metric structure. The generalized Tanaka-Webster connection on the contact metric manifold is very important because this connection preserves both the contact form and the compatible with the metric. Moreover, if it also preserves the almost complex structure, then we call it strictly pseudoconvex CR manifold. When it is given a spin$^c-$structure on the contact metric manifold, a spinorial connection on the associated spinor bundle  can be defined with the aid of the generalized Tanaka-Webster connection together with a connection on the determinant line bundle. These connections have been extensively studied in \cite{petit}. A Dirac operator associated to such a connection is formally self-adjoint for the natural inner product on the spinor bundle. Hence, we can define the Kohn-Dirac operator by the restriction to contact distribution.

The purpose of this article is to write down Seiberg-Witten equations on contact metric manifolds of dimension $5$ and give a global solution  for these equations on strictly pseudo-convex CR-manifolds of dimension $5$.

The paper is organized as follows. We begin with a section introducing the basic facts concerning the contact metric manifolds, the generalized Tanaka-Webster connection and CR structure. In the following section, we study self-dual $2-$forms on $5-$dimensional contact metric manifolds. In a fourth section, we discuss the spin$^c-$structures and the spinorial connections on the contact metric manifolds and Dirac type operators associated to the generalized Tanaka-Webster connection. In the final section we explain the curvature equation and the Dirac equation. Therefore, we write down Seiberg-Witten like equations on the contact metric manifolds of dimension $5$. In particular, we give a global solution for these equations on the strictly pseudoconvex CR manifolds.

\section{Contact Metric Manifolds and CR structure }

A contact form on a smooth manifold $M$ of dimension $(2n+1)$ is a $1-$form $\eta$ such that $\eta\wedge (d\eta)^n\neq 0$ everywhere on $M$. The contact form $\eta$ induces a hyperplane subbundle $H$ of the tangent bundle $TM$ given by $H=Ker\ \eta$. The Reeb vector field associated to $\eta$ is the vector field $\xi$ uniquely determined by $\eta(\xi)=1$ and $d\eta(\xi,.)=0$. Then, $(M,\eta)$ is called a contact manifold.

If $(M,\eta)$ is a contact manifold, then the pair $(H,d\eta|_{H})$ is a symplectic vector bundle. We fix an almost complex structure $J$ on $H$ compatible with $d\eta |_H$. Hence, we can define an associated hermitian metric on $H$ by $$g_{\eta,H}(X,Y)=d\eta(X,JY).$$ We can extend $J$ defined on $H$ to an endomorphism of the tangent bundle $TM$ by setting $J\xi=0$. Then, it holds $J^2=-Id+\eta\otimes\xi$. With this in mind, we extend $g_{\eta,H}$ to a Riemannian metric $g_{\eta}$ on $TM$ by setting $$g_{\eta}(X,Y)=d\eta(X,JY)+\eta(X)\eta(Y).$$ The metric $g_{\eta}$ is called Webster metric and said to be associated to $\eta$. Moreover, the following relations hold: $$g_{\eta}(\xi,X)=\xi(X), \ \ \  g_{\eta}(JX,Y)=d\eta(X,Y), \ \ \  g_{\eta}(JX,JY)=g_{\eta}(X,Y)-\eta(X)\eta(Y)$$ for any $X,Y\in \Gamma(TM).$ We call $(M,g_{\eta},\eta,\xi,J)$ a contact metric manifold. (For detailed information see \cite{Baum1, Baum2, petit})

Now we give a definition of the generalized Tanaka-Webster connection on the contact metric manifolds.

\begin{proposition}\cite{petit}
Let $(M,g_{\eta},\eta,\xi,J)$ be a contact metric manifold. Then, there exists a unique affine connection $\nabla$ on $TM$ with torsion $T$ such that
\begin{enumerate}
\item[a.] $\nabla\eta=0$, $\nabla\xi=0$
\item[b.] $\nabla g_{\eta}=0$
\item[c.] $T_H=-d\eta\otimes\xi$ and $T(\xi,\cdot)=\frac{1}{2}(J\circ \mathcal{L}_{\xi}J)$
\item[d.] For any $X,Y,Z\in \Gamma(TM)$ $$g_{\eta}((\nabla_X J)Y,Z)=\frac{1}{2}d\eta(X,N_H(Y,Z))$$ where $(\mathcal{L}_{\xi}J)X=[\xi,JX]-J[\xi,X]$ and $$N(Y,Z)=J^2[Y,Z]+[JY,JZ]-J[Y,JZ]-J[JY,Z]+d\eta(Y,Z)\xi.$$
\end{enumerate}
\end{proposition}
The endomorphism $\tau:=T(\xi,\cdot)$ is called the generalized Tanaka-Webster torsion. Moreover, $\tau$ is self-adjoint with respect to the Webster metric, that is, $$g_{\eta}(\tau X,Y)=g_{\eta}(X,\tau Y)$$ for any $X,Y\in \Gamma(TM)$ and satisfies $\tau(JX)=-J(\tau X)$.

In the final section, we deal with strictly pseudoconvex CR manifolds. Moreover, these manifolds are the special type CR manifolds. Then, now we give definitions of real CR structure and complex CR structure.

\begin{definition}
Let $M$ be a smooth connected manifold of dimension $2n+1$. A complex CR structure on $M$ is a complex subbundle $T_{10}$ of $TM^{\mathbb{C}}$ such that
\begin{enumerate}
\item[a)] dim$_{\mathbb{C}}T_{10}=n$

\item[b)] $T_{10}\cap\overline{T_{10}}=\{0\}$

\item[c)] $[\Gamma(T_{10}),\Gamma(T_{10})]\subset \Gamma(T_{10})$
\end{enumerate}

\end{definition}

\begin{definition}
Let $M$ be a smooth connected manifold of dimension $2n+1$. A real CR structure on $M$ is a pair $(H,J)$ such that

\begin{enumerate}
\item[a)] $H\subset TM$ is a real $2n-$dimensional subbundle,

\item[b)] $J:H\rightarrow H$ is a almost complex structure on $H$; $J^2=-Id$,

\item[c)] If $X,Y\in \Gamma(H)$ then, $[JX,Y]+[X,JY]\in \Gamma(H)$ and $$N(X,Y):=J([JX,Y]+[X,JY])-[JX,JY]+[X,Y]=0$$(Integrability condition)
\end{enumerate}
\end{definition}

If $T_{10}\subset TM^{\mathbb{C}}$ is a complex CR-structure, then $$H:=Re(T_{10}\oplus \overline{T_{10}})$$ and $$J(U+\overline{U}):=i(U-\overline{U})$$ defines a real CR-structure. Conversely, if the pair $(H,J)$ is a real CR-structure, then the eigenspace of the complex extension of $J$ on the complexification of $H$ to the eigenvalue $i$ is a complex CR-structure.
An odd dimensional manifold endowed with a real or complex CR structure is called CR manifold. (see also \cite{Baum1, Baum2})

Let an orientable CR manifold be given. Then, we globally define a form $\eta$ which is nonzero and vanishes on $H$. Moreover, there exists a global vector field $\xi$ in the complement of $H$ such that we regulate $\eta$ as $\eta(\xi)=1$ and $\eta(X)=0$ on $H$. We can define the Levi-form $L_{\eta}$ on $H$ as $$L_{\eta}(X,Y):=d\eta(X,JY)$$ for any $X,Y\in \Gamma(H)$. If the Levi-form is positive definite, then the CR manifold is called as strictly pseudoconvex CR manifold. This is equivalent to the integrability of $J$, i.e, $\nabla J=0$.  If $M$ is a strictly pseudoconvex CR manifold, then we define  a Riemannian metric on $M$ as $$g_{\eta}(X,Y)=L_{\eta}(X,Y)+\eta(X)\eta(Y)$$ for any $X,Y\in \Gamma(H)$. Take $\xi$ as the metric dual of $\eta$. Then we extend $J$ on $M$ by setting $J\xi=0$. Moreover, we have $g_{\eta}(JX,Y)=d\eta(X,Y)$. Since $\eta(H)=0$ and $\eta(\xi)=1$, we have $J^2(X)=-X+\eta(X)\xi$. Hence, $(M,g_{\eta},\eta,\xi,J)$ is a contact metric manifold.

Conversely, if $M$ is a contact metric manifold, then $(H,J|_H)$ satisfies the first two conditions of definition of a real CR structure. But for integrability condition we need to an extra condition. Namely, a contact metric manifold is a CR manifold if and only if the condition $$\eta([JX,Y])=-\eta([X,JY])$$ is satisfied for any $X,Y\in \Gamma(H)$.

\section{Self-dual $2-$forms on 5-dimensional Contact Metric Manifolds}

Let $(M,g_{\eta},\eta,\xi,J)$ be a $5-$dimensional contact metric manifold. For any $2-$form $\alpha \in \Omega^2(M)$ we have the splitting $\alpha=\alpha_H+\alpha_{\xi}$ where $\alpha_H=\alpha \circ \Pi$, $\Pi:TM\rightarrow H$ is the canonical projection and $\alpha_{\xi}=\eta\wedge i(\xi)\alpha$. Then, the decomposition of $\Omega^2(M)$ is given by $$\Omega^2(M)=\Omega_H^2(M)\oplus \eta\wedge\Omega_H^2(M)$$ where $\Omega_H^2(M)$ is the bundle of horizontal 2-forms. Moreover, any horizontal $2-$form can be split in its self-dual and anti-self dual parts as follows:

Let $\ast$ be the Hodge-star operator acting on the cotangent bundle $T^{\ast}M$. We can define the operator $$\star:\Omega^2(M)\rightarrow \Omega^2(M), \ \ \ \star(\beta):=\ast(\eta\wedge \beta).$$We can restrict the operator $\star$ to the space of horizontal $2-$forms $\Omega^2(H)$: $$\star:\Omega_H^2(M)\rightarrow \Omega_H^2(M), \ \ \ \star(\beta):=\ast(\eta\wedge \beta).$$ This operator satisfies $\star^2=id$. Then, we have the following orthogonal decomposition : $$\Omega_H^2(M)=\Omega_H^2(M)^+\oplus\Omega_H^2(M)^-$$ where $\Omega_H^2(M)^{\pm}$ is the eigenspace associated to eigenvalue $\pm 1$ of the operator $\star$. The eigenspace $\Omega_H^2(M)^+$ is called as the space of self-dual $2-$forms. In similar way, the eigenspace $\Omega_H^2(M)^-$ is called as the space of anti-self-dual $2-$forms. (see \cite{Bellettini, Fan1})

\section{Spin$^c-$structures and Dirac Operators on Contact Metric Manifolds }

\subsection{Spin$^c-$structures on Contact Metric Manifolds }
Let $(M,g_{\eta},\eta,\xi,J)$ be a contact metric manifold. The complexification $TM^ {\mathbb{C}}$ of the tangent bundle $TM$ splits into three parts: $$TM^ {\mathbb{C}}=T^ {1,0}M\oplus T^ {0,1}M\oplus \{\mathbb{C}\xi\},$$ where $T^ {1,0}M$ and $T^ {0,1}M$ are the subbundle associated to the eigenspace of the complex extension of $J$ on $H^ {\mathbb{C}}$ to the eigenvalues $i$ and $-i$, respectively.

\begin{proposition}{\cite{petit}}
Any contact metric manifold $(M^{2n+1},g_{\eta},\eta,\xi,J)$ admits a Spin$^c-$structure whose determinant bundle is $\wedge_H^{0,n}M$.

\end{proposition}

\begin{proposition}{\cite{petit}}
The spinor bundle $S$ can be identified with the bundle $\wedge_H^{0,*}M$ of the $(0,*)$ forms and the Clifford multiplication by a vector field $X$ on a $(0,q)$ form $\alpha$ is given by $$X\cdot \alpha=\sqrt{2}((X^ {0,1}_H)^*\wedge \alpha-(X_H^ {0,1})\lrcorner\alpha)+(-1)^{q+1}i\eta(X)\alpha$$ where $X_H$ is the horizontal part of $X$ and $\lrcorner$ is the contraction operator.
\end{proposition}

\subsection{Spinorial connections and Dirac type operators on contact metric manifolds}

Let $(M,g_{\eta},\eta,\xi,J)$ be a contact metric manifold equipped with a $Spin^c-$structure. Each unitary connection $A$ on $L$ induces a spinorial connection $\nabla^A$ on $S$ with the generalized Tanaka-Webster connection $\nabla$.

The Kohn-Dirac operator $D_H^A$ is defined as follows:
$$D_H^A=\displaystyle \sum_{i=1}^{2n} e_i\cdot \nabla^A_{e_i}$$ where $\{e_i\}$ is a local orthonormal frame of $H$.

The Dirac operator $D_A$ is defined by

\begin{equation*}
D_A=D_H^A+\xi\cdot \nabla^A_{\xi}.
\end{equation*}

Suppose that $M$ is a strictly pseudo convex CR manifold endowed with a spin$^c-$structure. Let $S=\wedge_H^{0,*}(M)\otimes \mathcal{L}$ be the spinor bundle. Each unitary connection $A$ on the determinant bundle $L$ induces a unitary connection on $\mathcal{L}$. We denote the associated covariant derivative on $\mathcal{L}$ by $\nabla^\mathcal{A}$. Let

$$\Omega_H^{0,*}(M;\mathcal{L})=\Gamma(\wedge_H^{0,*}(M)\otimes \mathcal{L}).$$ We can define the covariant derivative

$$\nabla_W^{\mathcal{A}^q}:\Omega_H^{0,q}(M;\mathcal{L})\rightarrow \Omega_H^{0,q}(M;\mathcal{L})$$ with the rule
$$\nabla_W^{\mathcal{A}^q}(\alpha\otimes z)=(\nabla_W^{q}\alpha)\otimes z+\alpha\otimes \nabla_W^{\mathcal{A}}z$$
where $\nabla_W^{q}\alpha$ is the natural extension of the generalized Tanaka-Webster connection to $\Omega_H^{0,q}(M)$, $z\in \Gamma(\mathcal{L})$ ve $W\in \Gamma(TM^{\mathbb{C}})$. The operators $\overline{\partial}_H^{\mathcal{A}^q}:\Omega_H^{0,q}(M;\mathcal{L})\rightarrow \Omega_H^{0,q+1}(M;\mathcal{L})$ and $\overline{\partial}_H^{{\mathcal{A}^q}^*}:\Omega_H^{0,q}(M;\mathcal{L})\rightarrow \Omega_H^{0,q-1}(M;\mathcal{L})$ are respectively given by

\begin{equation*}
\overline{\partial}_H^{\mathcal{A}^q}=\displaystyle \sum_{i=1}^{n}\overline{Z}_i^*\wedge \nabla_{\overline{Z}_i^*}^{\mathcal{A}^q}
\end{equation*}

\begin{equation*}
\overline{\partial}_H^{{\mathcal{A}^q}^*}=\displaystyle -\sum_{i=1}^{n}i(\overline{Z}_i)(\nabla_{Z_i}^{\mathcal{A}^q}).
\end{equation*}
From \cite{petit} we have on $\Omega_H^{0,*}(M;\mathcal{L})$ the following identities:

$$\nabla_W^A=\displaystyle \sum_{q=0}^d\nabla_W^{\mathcal{A}^q}$$
\begin{equation}\label{11}
\mathcal{D}_H^A=\sqrt{2}\displaystyle \sum_{q=0}^n (\overline{\partial}_H^{\mathcal{A}^q}+\overline{\partial}_H^{{\mathcal{A}^q}^*})
\end{equation} Moreover, by using the equation $\xi\cdot\nabla_{\xi}^A=\displaystyle \sum_{q=0}^n (-1)^{n+1}i\nabla_{\xi}^{\mathcal{A}^q}$ we obtain
\begin{equation*}
\begin{array}{lll}
    \mathcal{D}_A  & =  & \mathcal{D}_H^A+\xi\cdot\nabla_{\xi}^A\\
     \mathcal{D}_A   & = & \sqrt{2}\displaystyle \sum_{q=0}^n (\overline{\partial}_H^{\mathcal{A}^q}+\overline{\partial}_H^{{\mathcal{A}^q}^*}+(-1)^{n+1}i\nabla_{\xi}^{\mathcal{A}^q}).
\end{array}
\end{equation*}
(see also \cite{petit})

\section{Seiberg-Witten Like Equations on $5-$dimensional Contact metric manifolds}
In \cite{DK} Seiberg-Witten like equations on $5-$dimensional Euclidean space $\mathbb{R}^5$ were written. In this section, we will write Seiberg-Witten like equations on $5-$dimensional contact metric manifolds and give a solution to these equations on strictly pseudoconvex CR manifolds. 

For a spinor $\psi$ we define a $2-$form $\sigma(\psi)$ by the formula $$\sigma(\psi)(X,Y)=<X\cdot Y\cdot \psi,\psi>+<X,Y>|\psi|^2,$$ where $X,Y\in\Gamma(M)$. Note that $\sigma(\psi)$ is an imaginary valued $2-$form. $\sigma(\psi)$ can be also defined on $H$. Then, we call it $\sigma_H(\psi)$. Choose a local orthonormal frame field $\{e_1,e_2=J(e_1),e_3,e_4=J(e_3)\}$ of $H$. Then the expression of the $2-$form $\sigma_H(\psi)$ in the local frame field is given by
\begin{equation}\label{10}
\sigma_H(\psi)=\displaystyle \sum_{i<j}<e_ie_j\psi,\psi>e^i\wedge e^j.
\end{equation}

\begin{definition}
Let $(M^5,g_{\eta},\eta,\xi,J)$ be a contact metric $5-$manifold. Fix a spin$^c-$structure and a connection $A$ in the $U(1)-$principal bundle associated with the spin$^c-$structure. For $\psi\in \Gamma(S^+)$ Seiberg-Witten equations are defined by
\begin{equation}\label{12}
\begin{array}{l}
   D_A(\psi)=0       \\
   F_A^+=-\displaystyle\frac{1}{4}\sigma(\psi)^+
\end{array}
\end{equation}
where $F_A^+$ is the self-dual part of the curvature $F_A$ and $\sigma(\psi)^+$ is the self-dual part of the $2-$form $\sigma(\psi)$.

\end{definition}

Now we give a solution for Seiberg-Witten equations in dimension $5$. To do this we follow the method given in \cite{Friedrich}. From now on we suppose that $(M,g_{\eta},\eta,\xi,J)$ is a strictly pseudo-convex CR manifold.

$d\eta(X,Y):=g_{\eta}(JX,Y)$ is a $2-$form for any $X,Y\in \Gamma(TM)$. Choose a local orthonormal frame field $\{e_1,e_2=J(e_1),e_3,e_4=J(e_3),\xi\}$. Then the $2-$form $d\eta$ can be written in the following form: $$d\eta=e^1\wedge e^2+e^3\wedge e^4.$$ Let $(M,g_{\eta})$ be a contact metric manifold endowed with a spin$^c-$structure. Now we consider the following spin$^c$ structure $\kappa: \mathbb{C}l_5\rightarrow End(\mathbb{C}^4)$ coming from representation of the complex Clifford algebra $\mathbb{C}l_5$:

$\kappa(e_1)=\left(
               \begin{array}{cccc}
                 0 & i & 0 & 0 \\
                 i & 0 & 0 & 0 \\
                 0 & 0 & 0 & i \\
                 0 & 0 & i & 0 \\
               \end{array}
             \right)
$,
$\kappa(e_2)=\left(
               \begin{array}{cccc}
                 0 & 1 & 0 & 0 \\
                 -1 & 0 & 0 & 0 \\
                 0 & 0 & 0 & -1 \\
                 0 & 0 & 1 & 0 \\
               \end{array}
             \right)
$,

$\kappa(e_3)=\left(
               \begin{array}{cccc}
                 0 & 0 & 0 & 1 \\
                 0 & 0 & -1 & 0 \\
                 0 & 1 & 0 & 0 \\
                 -1 & 0 & 0 & 0 \\
               \end{array}
             \right)
$,
$\kappa(e_4)=\left(
               \begin{array}{cccc}
                 0 & 0 & 0 & i \\
                 0 & 0 & -i & 0 \\
                 0 & -i& 0 & 0 \\
                 i & 0 & 0 & 0 \\
               \end{array}
             \right)
$,

$\kappa(e_5)=\left(
               \begin{array}{cccc}
                 i & 0 & 0 & 0 \\
                 0 & -i & 0 & 0 \\
                 0 & 0 & i & 0 \\
                 0 & 0 & 0 & -i \\
               \end{array}
             \right)
$,
$\kappa(d\eta)=\left(
               \begin{array}{cccc}
                 0 & 0 & 0 & 0 \\
                 0 & 2i & 0 & 0 \\
                 0 & 0 & 0 & 0 \\
                 0 & 0 & 0 & -2i \\
               \end{array}
             \right)$

Let $(M,g_{\eta})$ be a contact metric manifold endowed with spin$^c-$structure. Then, the spinor bundle is $S=\wedge^ {0,*}_H(M)$. Namely,
$$S=\wedge^ {0,2}_H(M)\oplus \wedge^ {0,1}_H(M)\oplus \wedge^ {0,0}_H(M)$$
where $\wedge^ {0,2}_H(M)$ is the eigenspace corresponding to the eigenvalue $2i$ of the mapping $\kappa(d\eta):S\rightarrow S$ and has dimension one and $\wedge^ {0,1}_H(M)$ is the eigenspace corresponding to the eigenvalue $0$ of the mapping $\kappa(d\eta):S\rightarrow S$ and has dimension $2$ and $\wedge^ {0,0}_H(M)$ is the eigenspace corresponding to the eigenvalue $-2i$ of the mapping $\kappa(d\eta):S\rightarrow S$ and has dimension one.

If $\psi_0\in \wedge^ {0,0}_H(M)$, then $\psi_0$ denotes the spinor corresponding the constant function $1$. In the chosen coordinates
$$\psi_0=\left(\begin{array}{c}
      0    \\
      0   \\
      0\\
      1
\end{array}\right)$$
Moreover, we have $d\eta\cdot \psi_0=-2i\psi_0$. It is given in (\ref{10}) the expression of $\sigma_H(\psi)$ in the local coordinates. By using this we obtain the following identity:
\begin{equation}\label{13}
\sigma_H(\psi_0)=-id\eta
\end{equation}

When $M$ is a strictly pseudo-convex CR manifold, $M$ has also a complex CR structure. Let $s=(Z_1,\ldots,Z_n)$ be a local unitary basis of $T_{10}$ over $U\subset M$. Let us denote by $\omega_s:=(\omega_{\alpha\beta})$ the matrix of connection form of $\nabla$ with respect to frame $s$. Then, we can write in the following way:

 $$\nabla Z_{\alpha}=\displaystyle \sum_{\beta}\omega_{\alpha\beta}Z_{\beta}.$$
$\{Z_1,\ldots,Z_n,\overline{Z}_1,\ldots,\overline{Z}_n,\xi\}$ is a local frame of the complexified tangent bundle $TM^{\mathbb{C}}$ over $U$. Let $\{\theta^1,\ldots,\theta^n,\overline{\theta}^1,\ldots,\overline{\theta}^n,\eta\}$ be the corresponding dual basis. Then,
$$\tau_s=\overline{\theta}^1\wedge\ldots\wedge\overline{\theta}^n:U\rightarrow \Lambda_H^{0,n}(M)$$
is a local section in determinant line bundle $\Lambda_H^{0,n}(M)$. The Webster connection $\nabla$ defines a covariant derivative in the canonical line bundle $\Lambda_H^{0,n}(M)$ such that

$$\nabla(\overline{\theta}^1\wedge\ldots\wedge\overline{\theta}^n)=-Tr(\omega_s)\overline{\theta}^1\wedge\ldots\wedge\overline{\theta}^n$$
Since $\nabla$ is a metric with respect to $g_{\eta}$, the trace $Tr(\omega_s)$ is purely imaginary. Therefore, this connection $\nabla$ in $\Lambda_H^{0,n}(M)$ induce a connection on the associated $S^1-$principal bundle $P_{S^1}$. Let us denote this connection by $A$. Then, $$\tau_s^*A=-Tr \ \overline{\omega_s}=Tr\omega_s$$
is a local connection forms on $S^1-$bundle $P_1$. Let $F_A$ be the curvature form of the connection $A$. The curvature form $F_A$ is a $2-$form on $M$ with values in i$\mathbb{R}$. Over $U\subset M$ we have
\begin{equation}\label{31}
  F_A=dA=Tr d\omega_s
\end{equation}
Moreover,
\begin{equation}
  Ric(X,Y)=Tr(d\omega_s)-Tr(\omega_s\wedge \omega_s)=Trd\omega_s.
\end{equation}
From (\ref{31}) it follows
\begin{equation}\label{5}
  F_A=Ric.
\end{equation}

We can define $(4,0)$ curvature tensor $\mathcal{R}$ as follows:
$$\mathcal{R}(X,Y,Z,V)=g_{\eta}(R(X,Y)Z,\overline{V})$$
for $X,Y,Z,V\in TM^{\mathbb{C}}$. The curvature tensor $\mathcal{R}$ has the following properties:

\begin{proposition}\label{16}\cite{Baum2}
Let $X,Y,Z,V\in TM^{\mathbb{C}}$ and $A,B,C,D\in T_{10}$. Then, we have
\begin{enumerate}
  \item $\mathcal{R}(X,Y,Z,V)=-\mathcal{R}(Y,X,Z,V)=-\mathcal{R}(X,Y,V,Z)$
  \item $\overline{\mathcal{R}(X,Y,Z,V)}=\mathcal{R}(\overline{X},\overline{Y},\overline{Z},\overline{V})$
  \item $\mathcal{R}(A,\overline{B},C,\overline{D})=\mathcal{R}(C,\overline{B},A,\overline{D})$
  \item $\mathcal{R}(A,B,\cdot,\cdot)=0$
\end{enumerate}
\end{proposition}

The $(2,0)-$tensor field
$$Ric:=\displaystyle\sum_{\alpha=1}^n\mathcal{R}(\cdot,\cdot,Z_{\alpha},\overline{Z_{\alpha}})$$
is called Webster-Ricci tensor. Moreover, the function $$s=trace Ric$$ is called Webster scalar curvature. That is, the scalar curvature $s$ is given by
$$\displaystyle\sum_{\alpha=1}^n\mathcal{R}(Z_{\alpha},\overline{Z_{\alpha}},Z_{\alpha},\overline{Z_{\alpha}})$$
From Proposition $(\ref{16})$ it can be seen that
\begin{enumerate}
  \item For $X,Y\in \Gamma(H)$ we have \begin{equation}Ric(JX,JY)=Ric(X,Y)\end{equation}
  \item $Ric(X,Y)\in i\mathbb{R}$ for $X,Y\in \Gamma(TM)$
  \item $s$ is a real function.
\end{enumerate}

\begin{proposition}
Let $\rho_H$ be a ricci form on $H$ and $s$ be a scalar curvature. Then the following identity is holds: $$\rho_H^+=\displaystyle -\frac{s}{4}d\eta$$
where $\rho_H^+$ is a the self-dual part of the ricci form $\rho_H$.
\end{proposition}

\begin{proof}
In local coordinates the almost complex structure $J$ and the Ricci tensor are respectively given by:
  \begin{equation*}
  J=\left(
  \begin{array}{ccccc}
    0 & -1 & 0 & 0 & 0   \\
    1 & 0 & 0 & 0 & 0   \\
    0 & 0 & 0 & -1 & 0  \\
    0 & 0 & 1 & 0 & 0  \\
    0 & 0 & 0 & 0 & 0  \\
  \end{array}
\right),\ \ Ric=\left(
  \begin{array}{cccccccc}
    R_{11} & R_{12} & R_{13} & R_{14} & R_{15}  \\
    R_{12} & R_{22} & R_{23} & R_{24} & R_{25}  \\
    R_{13} & R_{23} & R_{33} & R_{34} & R_{35}  \\
    R_{14} & R_{24} & R_{34} & R_{44} & R_{45} \\
    R_{15} & R_{25} & R_{35} & R_{45} & R_{55}\\
  \end{array}
\right)
\end{equation*}
Since $J\circ Ric=Ric\circ J$, we obtain the following skew symmetric endomorphism:
$$J\circ Ric=\left(
  \begin{array}{cccccccc}
    0 & -R_{11} & -R_{23} & -R_{24} & 0  \\
    R_{11} & 0 & R_{24} & -R_{23} & 0  \\
    R_{23} & -R_{24} & 0 & -R_{33} & 0 \\
    R_{24} & R_{23} & R_{33} & 0 &0  \\
    0 & 0 & 0 & 0 & 0  \\
  \end{array}
\right)$$ with the conditions:
$$\begin{array}{lll }
     R_{15}=R_{25}=R_{35}=R_{45}=R_{55}=0 &  & R_{12}=R_{34}=0 \\
      R_{11}=R_{22}&   &R_{14}=-R_{23}\\
R_{24}=R_{13}& &R_{33}=R_{44}
\end{array}$$
The ricci form is given by $\rho_H(X,Y)=g_{\eta,H}(X,JRicY)$. Moreover, the ricci form $\rho_H$ can be written in the following way:
$$\rho_H=-R_{11}e_1\wedge e_2-R_{33}e_3\wedge e_4-R_{24}(e_1\wedge e_4-e_2\wedge e_3)-R_{23}(e_1\wedge e_3+e_2\wedge e_4).$$
Since the $2-$forms $e_1\wedge e_4-e_2\wedge e_3$ and $e_1\wedge e_3+e_2\wedge e_4$ are anti-self-dual $2-$forms, the self-dual part of $\rho_H$ is given by
\begin{equation}\label{1}
\rho^+_H=\frac{-R_{11}-R_{33}}{2}d\eta=-\frac{R_{11}+R_{22}+R_{33}+R_{44}}{4}d\eta=-\frac{s}{4}d\eta
\end{equation}
\end{proof}

\begin{proposition}\label{18}
Let $X,Y\in \Gamma(H)$. $\rho_H$ is a ricci form on $H$, that is, $\rho_H(X,Y)=Ric(X,JY)$. Then we have the following identity:
$$Ric(X,Y)=\displaystyle i\rho_H(X,Y)+\frac{i}{2}\sum_{\alpha=1}^{4}g_{\eta}(B_{\alpha}(X,Y),Je_{\alpha})$$

\end{proposition}

\begin{proof}
$$\begin{array}{ccl}
   Ric(X,Y)   & = & \displaystyle \sum_{\alpha=1}^{n}\mathcal{R}(X,Y,Z_{\alpha},\overline{Z}_{\alpha}) \\
      & = &  \displaystyle \sum_{\alpha=1}^{n}g_{\eta}(R(X,Y)Z_{\alpha},Z_{\alpha})\\
& = & \displaystyle \sum_{\alpha=1}^{n}g_{\eta}(R(X,Y)\frac{1}{\sqrt{2}}(e_{\alpha}-iJe_{\alpha}),\frac{1}{\sqrt{2}}(e_{\alpha}-iJe_{\alpha}))\\
& = & \displaystyle \frac{1}{2}\sum_{\alpha=1}^{n}g_{\eta}(R(X,Y)(e_{\alpha}-iJe_{\alpha}),(e_{\alpha}-iJe_{\alpha}))\\
& = & \displaystyle \frac{1}{2}\sum_{\alpha=1}^{n}i g_{\eta}(R(X,Y)e_{\alpha},Je_{\alpha})-ig_{\eta}(R^W(X,Y)Je_{\alpha},e_{\alpha})\\
& = & \displaystyle \frac{1}{2}\sum_{\alpha=1}^{n}i g_{\eta}(R(X,Y)e_{\alpha},Je_{\alpha})+ig_{\eta}(R(X,Y)e_{\alpha},Je_{\alpha})\\
& = & \displaystyle i\sum_{\alpha=1}^{n}g_{\eta}(R(X,Y)e_{\alpha},Je_{\alpha})\\
\end{array}$$

For $X,Y\in \Gamma(H)$ we have
\begin{equation}\label{1}
Ric(X,Y)= \displaystyle i\sum_{\alpha=1}^{n}g_{\eta}(R(X,Y)e_{\alpha},Je_{\alpha})=\displaystyle \frac{i}{2}\sum_{\alpha=1}^{2n}g_{\eta}(R(X,Y)e_{\alpha},Je_{\alpha})
\end{equation}
For any $X,Y,Z\in \Gamma(H)$ the curvature $R$ of $\nabla$ satisfies the following first Bianchi identity:
\begin{equation*}
R(X,Y)Z+R(Z,X)Y+R(Y,Z)X=d\eta(X,Y)\tau(Z)+d\eta(Z,X)\tau(Y)+d\eta(Y,Z)\tau(X)
\end{equation*}

Here, if we write $Z=e_{\alpha}$, then we obtain the following identity:
\begin{equation}\label{2}
R(X,Y)e_{\alpha}+R(e_{\alpha},X)Y+R(Y,e_{\alpha})X=d\eta(X,Y)\tau(e_{\alpha})+d\eta(e_{\alpha},X)\tau(Y)+d\eta(Y,e_{\alpha})\tau(X)\end{equation}

Let $B_{\alpha}(X,Y)$ be the right side of above identity, that is,

$$B_{\alpha}(X,Y)=d\eta(X,Y)\tau(e_{\alpha})+d\eta(e_{\alpha},X)\tau(Y)+d\eta(Y,e_{\alpha})\tau(X).$$

Using (\ref{1}) and (\ref{2}), we have
$$\begin{array}{ccl}
   Ric(X,Y)   & = & \displaystyle \frac{i}{2}\sum_{\alpha=1}^{4}g_{\eta}(R(X,Y)e_{\alpha},Je_{\alpha}) \\
      & = & \displaystyle \frac{i}{2}\sum_{\alpha=1}^{4}g_{\eta}(-R(e_{\alpha},X)Y-R(Y,e_{\alpha})X-B_{\alpha}(X,Y),Je_{\alpha}) \\
& = & \displaystyle \frac{i}{2}\sum_{\alpha=1}^{4}g_{\eta}(R(e_{\alpha},X)JY,e_{\alpha})-g_{\eta}(R(e_{\alpha},Y)JX,e_{\alpha})+g_{\eta}(B_{\alpha}(X,Y),Je_{\alpha}) \\
& = & \displaystyle \frac{i}{2} (Ric(X,JY)-Ric(Y,JX))+\frac{i}{2}\sum_{\alpha=1}^{4}g_{\eta}(B_{\alpha}(X,Y),Je_{\alpha}) \\
& = & \displaystyle \frac{i}{2} (Ric(X,JY)+Ric(X,JY))+\frac{i}{2}\sum_{\alpha=1}^{4}g_{\eta}(B_{\alpha}(X,Y),Je_{\alpha}) \\
& = & \displaystyle iRic(X,JY)+\frac{i}{2}\sum_{\alpha=1}^{4}g_{\eta}(B_{\alpha}(X,Y),Je_{\alpha}) \\
& = & \displaystyle i\rho_H(X,Y)+\frac{i}{2}\sum_{\alpha=1}^{4}g_{\eta}(B_{\alpha}(X,Y),Je_{\alpha}) \\
\end{array}$$

\end{proof}

Let $B(X,Y)=\displaystyle\frac{i}{2}\sum_{\alpha=1}^{4}g_{\eta}(B_{\alpha}(X,Y),Je_{\alpha})$  be the second term in the last proposition. In this case, we express the last proposition as follows:
\begin{equation}\label{14}
Ric(X,Y)=\displaystyle i\rho_H(X,Y)+B(X,Y)
\end{equation}
$B(X,Y)$ is a $2-$form because we have $B_{\alpha}(X,Y)=-B_{\alpha}(Y,X)$, that is, $B(X,Y)=-B(Y,X)$.

\begin{proposition}\label{19}
The $2-$form $B(X,Y)=\displaystyle\frac{i}{2}\sum_{\alpha=1}^{4}g_{\eta}(B_{\alpha}(X,Y),Je_{\alpha})$ vanishes.
\end{proposition}
\begin{proof}
It is easy to verify this by direct calculation.
\end{proof}

From Proposition (\ref{18}) and Proposition (\ref{19}) we deduce that
\begin{equation}\label{20}
Ric(X,Y)=\displaystyle i\rho_H(X,Y)
\end{equation}
where  $\rho_H$ is a ricci form on $H$, namely,  $\rho_H(X,Y)=Ric(X,JY)$ for any $X,Y\in \Gamma(H)$.

Let $(M^5,g_{\eta},\eta,\xi,J)$ be strictly pseudo-convex contact manifold of dimension $5$. Suppose that the scalar curvature $s$ of $M$ is negative and constant. Then, let $\psi=\sqrt{-s}\psi_0$. In this case, $\psi\in \wedge_H^{0,0}(M)$. From $(\ref{13})$ we have
\begin{equation}\label{22}
\sigma_H(\psi)=isd\eta
\end{equation}
By using (\ref{5}),(\ref{14}), (\ref{20}) and (\ref{22}) we obtain
\begin{equation}
F_{A}^+=Ric^+=i\rho^+=-i\displaystyle\frac{s}{4}d\eta=-\displaystyle\frac{1}{4}\sigma_H(\psi)=-\displaystyle\frac{1}{4}\sigma(\psi)^+.
\end{equation}
Moreover, since $\psi=\sqrt{-s}\psi_0\in \wedge_H^{0,0}(M)$ and the spinor $\psi_0$ is a spinor corresponding to the constant function $1$ we deduce that
\begin{equation}
\mathcal{D}_H^A \psi=0.
\end{equation}
Moreover, $$\mathcal{D}^A \psi=0.$$ The pair $(A,\psi=\sqrt{-s}\psi_0)$ is a solution of Seiberg-Witten like equations in (\ref{12}).

\end{document}